\documentclass[12pt,reqno]{amsart}

\usepackage{amssymb,latexsym}
\usepackage{amsmath}
\usepackage{amsthm}
\usepackage{graphicx}
\usepackage{hyperref}
\usepackage{titletoc}
\usepackage{pdfsync}
\usepackage{color}
\usepackage{cite}

\numberwithin{equation}{section}

\newtheorem{theorem}{Theorem}[section]

\newtheorem{lemma}[theorem]{Lemma}

\theoremstyle{definition}

\newtheorem{remark}{Remark}

\def\R{{\Bbb R}}

\def\begeq{\begin{equation}}
\def\endeq{\end{equation}}

\allowdisplaybreaks

\begin{document}

\title[Some semilinear elliptic equations with Robin boundary condition]{Asymptotic behavior of solutions to elliptic problems with Robin boundary conditions}

\author{Mengyao Chen}
\address{Mengyao Chen, School of Mathematics and Systems Science \& Hubei Province Key Laboratory of Systems Science in Metallurgical Process, Wuhan University of Science and Technology, Wuhan 430065, P. R. China.}
\email{cmy@mails.ccnu.edu.cn}

\author{Massimo Grossi}
\address{Massimo Grossi, Dipartimento di Matematica, Universit\`a di Roma ``La Sapienza'', P.le A. Moro 2, 00185 Roma, Italy}
\email{massimo.grossi@uniroma1.it}

\author{Qi Li}
\address{Qi Li, School of Mathematics and Systems Science \& Hubei Province Key Laboratory of Systems Science in Metallurgical Process, Wuhan University of Science and Technology, Wuhan 430065, P. R. China.}
\email{qili@mails.ccnu.edu.cn}

\begin{abstract}
In this paper, we investigate the asymptotic behavior, as $\beta \to 0$, of positive solutions to the semilinear elliptic Robin problem
\begin{equation*}
\begin{cases}
-\Delta u = u^p, & \text{in } \Omega,\\
u > 0, & \text{in } \Omega,\\
\frac{\partial u}{\partial \nu} + \beta u = 0, & \text{on } \partial \Omega,
\end{cases}
\end{equation*}
where $p \ge 0$, $\beta > 0$, and $\Omega$ is a bounded smooth domain.

We will prove that, for all $p\ge0$, the solution $u_\beta$ behaves like a constant as $\beta\to0$. However, the value of this constant is strongly influenced by the value of $p$. Indeed,
\begin{itemize}
\item if $0 \le p < 1$, $u_\beta$ blows up uniformly in $\Omega$ as $\beta \to 0$.
\item if $p=1$ (eigenvalue problem),  $u_\beta$  converge to a constant.
\item if $p>1$ $u_\beta$  converge uniformly to zero.
\end{itemize}
In the critical and supercritical regime $p \ge \frac{N+2}{N-2}$, the existence of solutions is no longer guaranteed a priori. In this case, when $\Omega$ is a ball and $0<\beta<\frac{2}{p-1}$ we prove the existence of a radial positive solution.
\medskip


\medskip


\end{abstract}

\date{}
\maketitle

\section{Introduction and main results}

Let $\Omega$ be a bounded smooth domain in $\R^N$, with $N\ge2$, and let $\nu$ denote the unit outward normal to $\partial\Omega$. In this paper we study the semilinear elliptic problem with Robin boundary conditions
\begin{equation}\label{eq0.1}
\begin{cases}
-\Delta u=u^p,\quad &\text{in }\Omega,\\
u>0,\quad &\text{in }\Omega,\\
\frac{\partial u}{\partial \nu}+\beta u=0,\quad &\text{on }\partial\Omega,
\end{cases}
\end{equation}
where $p\ge0$ and $\beta>0$.

Throughout the paper, we assume that when $p=1$ the problem \eqref{eq0.1} is replaced by the eigenvalue problem
\begin{equation}\label{eig}
\begin{cases}
-\Delta u=\lambda_{1,\beta}u,\quad &\text{in }\Omega,\\
u>0,\quad &\text{in }\Omega,\\
\frac{\partial u}{\partial \nu}+\beta u=0,\quad &\text{on }\partial\Omega.
\end{cases}
\end{equation}
This allows us to state the results in a unified way for all $p \ge 0$.

We work in the Sobolev space
\[
H^1(\Omega):=\{u\in L^2(\Omega): \nabla u\in L^2(\Omega)\},
\]
endowed with the norm
\[
\|u\|=\left(\int_{\Omega}\bigl(|\nabla u|^2+u^2\bigr)\,dx\right)^{1/2}.
\]
For each fixed $\beta>0$, we say that $u\in H^1(\Omega)$ is a weak solution of \eqref{eq0.1} if
\[
\int_{\Omega}\nabla u\nabla v\,dx+\beta\int_{\partial\Omega}uv\,dS
=\int_{\Omega}u^p v\,dx,
\qquad \forall v\in H^1(\Omega).
\]
Let $u_+=\max\{u,0\}$. In order to find a weak solution to \eqref{eq0.1} for $0\le p\le \frac{N+2}{N-2}$, it is enough to look for nontrivial critical points of the functional
\[
J_\beta(u)=\frac12\int_{\Omega}|\nabla u|^2\,dx+\frac{\beta}{2}\int_{\partial\Omega}u^2\,dS
-\frac{1}{p+1}\int_{\Omega}u_+^{p+1}\,dx.
\]
The corresponding solutions $u_\beta$ is a weak solution and, by elliptic regularity, a classical one.
\subsection{Background}

Robin boundary conditions interpolate between the classical Dirichlet and Neumann cases, which formally correspond to the limiting regimes $\beta=\infty$ and $\beta=0$. There is a vast literature on these two problems. The existence of solutions and their qualitative properties, such as uniqueness, nondegeneracy, and asymptotic behavior, have been widely investigated over the last few decades. Given the breadth of the subject, we do not attempt a complete survey and refer the reader, for instance, to \cite{BN,BO,GGP,GILY,GNN,GST,W,Z}  and the references therein.

Compared with the Dirichlet and Neumann problems, the Robin case has been less extensively studied, although several existence and qualitative results are available in the literature. We refer, among others, to \cite{D1,D2,BG,DF,FD,W,vB,GS} for results concerning eigenvalue estimates, existence, uniqueness, and asymptotic properties under Robin boundary conditions.

A first important distinction concerns the range of the exponent $p$. The behavior of the problem is substantially different in the sublinear, linear, subcritical, critical, and supercritical regimes, and it is useful to briefly review what is known in each case.

\subsection{Known results}

We begin with the case $p=0$, for which \eqref{eq0.1} reduces to the torsion problem
\begin{equation}\label{1.3}
\begin{cases}
-\Delta u=1,\quad &\text{in }\Omega,\\
\frac{\partial u}{\partial \nu}+\beta u=0,\quad &\text{on }\partial\Omega.
\end{cases}
\end{equation}
By the maximum principle, the solution is unique. In \cite{vB} it was proved that $u_\beta$ is bounded if and only if $\lambda_{1,\beta}>0$, where $\lambda_{1,\beta}$ denotes the first eigenvalue of $-\Delta$ under Robin boundary conditions. In addition, the estimate
\begin{equation}\label{1.4}
\lambda_{1,\beta}^{-1}\leq \|u_\beta\|_{L^\infty(\Omega)}\leq
6N\lambda_{1,\beta}^{-1}\log\bigl(2^{11}3\sqrt{3}N(1+\beta^{-1}\sqrt{\lambda_{1,\beta}})\bigr)
\end{equation}
was obtained.

The range $0\le p<1$ corresponds to a sublinear nonlinearity and essentially exhibits the same phenomena as the torsion problem. In this regime the equation has a regularizing character, and one expects positive solutions to exist for every $\beta>0$. At the same time, as $\beta\to0$, the Robin condition approaches the Neumann one, and this suggests the onset of a blow-up phenomenon. One of the main goals of this paper is precisely to describe this behavior and to establish the exact leading-order asymptotics of positive solutions in this range of exponents.

When $p=1$, problem \eqref{eq0.1} becomes closely related to the Robin eigenvalue problem
\begin{equation}\label{1.5}
\begin{cases}
-\Delta u=\lambda_{1,\beta}u,\quad &\text{in }\Omega,\\
u>0,\quad &\text{in }\Omega,\\
\frac{\partial u}{\partial \nu}+\beta u=0,\quad &\text{on }\partial\Omega.
\end{cases}
\end{equation}
For each $\beta>0$, Daners \cite{D1} proved the existence of a positive lower bound for $\lambda_{1,\beta}$, depending on $N$, $\beta$, and $|\Omega|$. Indeed,  by Lemma \ref{lemma2.2} we have  
$$
\lambda_{1,\beta}\le\beta\frac{|\partial\Omega|}{|\Omega|},
$$
which implies that $\lambda_{1,\beta}\to0$ as $\beta\to0$.
Later, it was shown in \cite{D2} that
\[
\lambda_{1,\beta}\ge \bar{\lambda}_1,
\]
where $\bar{\lambda}_1$ is the first Robin eigenvalue of a ball having the same measure as $\Omega$. Bucur and Giacomini \cite{BG} obtained related results for sublinear problems, showing in particular how the spectral properties of the Robin Laplacian influence the qualitative behavior of solutions when $0\le p<1$.

For subcritical exponents,
\[
1<p<\frac{N+2}{N-2},
\]
one may consider the variational quantity
\begin{equation}\label{defIbeta}
I_\beta=\inf_{u\in H^1(\Omega)\setminus\{0\}}
\frac{\int_{\Omega}|\nabla u|^2\,dx+\beta\int_{\partial\Omega}u^2\,dS}
{\left(\int_{\Omega}|u|^{p+1}\,dx\right)^{\frac{2}{p+1}}}.
\end{equation}
By arguments similar to those in \cite{Z}, one can verify that $I_\beta$ is achieved  by the compactness of the Sobolev embeddings
, and then a suitable scaling yields a positive solution of \eqref{eq0.1}. On the other hand, uniqueness is more delicate and  in general depending on the shape of the domain. In \cite{DF}, using a priori estimates together with a careful asymptotic analysis, Dai and Fu proved uniqueness of positive solutions when $\beta$ is sufficiently small. In \cite{FD}, they studied the problem in annular domains and showed that the positive solution is unique for $\beta$ small enough, while uniqueness fails for $\beta$ sufficiently large. The question remains open for arbitrary $\beta>0$, even in the case where $\Omega$ is a ball. We also note that when $\beta$ is large, the Robin condition approaches the Dirichlet one; therefore, in domains where uniqueness is known for the corresponding Dirichlet problem, one may expect uniqueness for sufficiently large $\beta$ as well. 

For the critical exponent
\[
p=2^*-1,\qquad 2^*=\frac{2N}{N-2},
\]
the embedding $H^1(\Omega)\hookrightarrow L^{2^*}(\Omega)$ is no longer compact, and therefore the functional $J_\beta$ does not satisfy the Palais--Smale condition globally. In \cite{W}, Wang proved that $J_\beta$ satisfies the $(PS)_c$ condition provided
\[
0<c<\frac{1}{2N}\mathcal{S}^{N/2},
\]
where $\mathcal{S}$ denotes the best Sobolev constant for the embedding $D^{1,2}(\R^N)\hookrightarrow L^{2^*}(\R^N)$. As a consequence, he showed that \eqref{eq0.1} admits a positive solution if $\beta$ is sufficiently small. In addition, when $\Omega=B_1(0)$, by means of the Pohozaev identity he proved that \eqref{eq0.1} admits a radial solution if and only if $\beta\in(0,N-2)$.

We now turn to the supercritical regime
\[
p>2^*-1.
\]
To the best of our knowledge, there are no general existence results for the supercritical Robin problem \eqref{eq0.1}. Motivated by the ideas in \cite{W}, and also by related results of Serrin and Zou \cite{SZ}, we prove an existence result in the radial setting when $\Omega$ is a ball.

\subsection{Main results}
The main purpose of this paper is to study the asymptotic behavior of positive solutions of \eqref{eq0.1} as $\beta\to0$ for all $p\ge0$. According to the previous discussion, existence is known in the relevant ranges at least for $\beta>0$ sufficiently small, and our analysis describes the precise asymptotic profile of such solutions.

Before presenting the asymptotic results, we first establish the uniqueness of positive solutions.  
\begin{theorem}\label{Theorem0.1}
Suppose that $p\ge 0$. We have,
\begin{enumerate}
\item  if $0\le p<1$ then the problem \eqref{eq0.1} has exactly one positive solution for each $\beta>0$;
\item if $p=1$ then the eigenvalue problem has  exactly one positive solution for  each $\beta>0$ and $\|u_\beta\|_{L^\infty(\Omega)}=1$;
\item if $1<p<2^*-1$ then the problem \eqref{eq0.1} has exactly one positive solution for $\beta$ small enough;
\item if $p=2^*-1$, then equation \eqref{eq0.1} has exactly  one ground state solution for $\beta$ small enough;
\item  if $p\ge2^*-1$ and $\Omega$ is a ball, then equation \eqref{eq0.1} has exactly  one positive radial solution for $\beta$ small enough.
\end{enumerate}
\end{theorem}
We now begin to discuss the asymptotic behavior of the solution $u_\beta$ for small $\beta$ corresponding to the various values of $p$.

First let us consider $p\in[0,1]$. Then the unique solution $u_\beta$ of \eqref{eq0.1} satisfies
$$
u_\beta\sim C(\beta),
$$
where $C(\beta)$ is a constant that tends to $0$, $1$, or infinity depending on the value of $p$.
Below is the precise statement.
\begin{theorem}\label{Theorem1.2}
Suppose that $p\in[0,1)$ and $N\ge2$. If $u_\beta$ is a positive solution to equation \eqref{eq0.1}, then
\[
u_\beta=
|\Omega|^{\frac{1}{1-p}}|\partial\Omega|^{-\frac{1}{1-p}}
\beta^{-\frac{1}{1-p}}
+o\bigl(\beta^{-\frac{1}{1-p}}\bigr),
\qquad \text{as }\beta\to0,
\]
where $|\Omega|$ and $|\partial\Omega|$ denote the volume of $\Omega$ and the surface measure of $\partial\Omega$, respectively.
\end{theorem}

\begin{remark}
For $p=0$, it was proved in \cite{GS} that
\begin{equation}\label{1.10}
\lim_{\beta\to0}\frac{\lambda_{1,\beta}}{\beta}=\frac{|\partial\Omega|}{|\Omega|}.
\end{equation}
Combining \eqref{1.4} and \eqref{1.10}, one obtains
\[
|\Omega|\,|\partial\Omega|^{-1}\beta^{-1}
\le \|u_\beta\|_{L^\infty(\Omega)}
\le \frac{C|\log\beta|}{\beta},
\qquad \text{as }\beta\to0.
\]
We would like to point out that we can give a more precise leading order estimate.
Thus, Theorem \ref{Theorem1.2} improves the estimate obtained in \cite{vB}.
\end{remark}

The linear case requires a separate discussion where the associated first Robin eigenfunction is considered.

\begin{theorem}\label{Theorem1.3}
Suppose that $u_\beta>0$ is the first eigenfunction of
\begin{equation*}
\begin{cases}
-\Delta u=\lambda_{1,\beta}u,\quad &\text{in }\Omega,\\
\frac{\partial u}{\partial \nu}+\beta u=0,\quad &\text{on }\partial\Omega,\\
\|u\|_{L^\infty(\Omega)}=1.
\end{cases}
\end{equation*}
Then
\[
u_\beta\to1,
\qquad \text{as }\beta\to0.
\]
\end{theorem}

Next we consider the superlinear case.

Although our results will hold for all $p>1$, it is important to emphasize that there is a significant difference between the subcritical case $1<p<\frac{N+2}{N-2}$ and the critical and supercritical case $p\ge\frac{N+2}{N-2}$. Indeed, while in the subcritical case it is fairly standard to deduce the existence of a solution of \eqref{eq0.1}, in the critical and supercritical case, the existence of solutions is a more delicate issue. Indeed, the lack of compactness (in the critical case), or even of the Sobolev embedding (in the supercritical case), prevents the use of standard variational methods. In this framework, our first result concerns the existence of a solution for $p>\frac{N+2}{N-2}$ in the case of the ball.
\begin{theorem}\label{Theorem1.1}
Suppose that
\[
p>\frac{N+2}{N-2},\qquad 0<\beta<\frac{2}{p-1},
\]
and assume that $\Omega$ is a ball. Then equation \eqref{eq0.1} admits a radial positive solution.
\end{theorem}

The proof of Theorem \ref{Theorem1.1} will be given in Section \ref{s5}. The problem of the existence of solutions to \eqref{eq0.1} in general domains is far from being understood.

On the other hand, assuming the existence of a solution $u_\beta$ in $\Omega$ in the supercritical case, the asymptotic behavior turns out to be the same as in the subcritical case,

\begin{theorem}\label{Theorem1.4}
Suppose that $p>1$, $N\ge2$ and $u_\beta$ is a positive solution to equation \eqref{eq0.1}. 
If one of the following three conditions hold:
\begin{enumerate}
\item $1<p<2^*-1$;
\item $p=2^*-1$ and $u_\beta$ is a ground state solution;
\item $p\ge 2^*-1$, $\Omega$ is a ball and $u_\beta$ is radial.
\end{enumerate}
Then we have 
\[
u_\beta=
|\Omega|^{-\frac{1}{p-1}}|\partial\Omega|^{\frac{1}{p-1}}
\beta^{\frac{1}{p-1}}
+o\bigl(\beta^{\frac{1}{p-1}}\bigr),
\qquad \text{as }\beta\to0.
\]
\end{theorem}

The paper is organized as follows. In Section 2 we collect some preliminary results. Then we study the existence, uniqueness and asymptotic behavior of positive solutions in Section 3, Section 4 and Section 5 respectively.

\section{Preliminaries}
In this section, we provide some preliminary results which will be used in the latter.

\begin{lemma}\label{lemma2.3}
(Lemma 5.2, \cite{W}) Suppose that $\partial\Omega$ is sufficiently smooth, $f\in L^p(\Omega)$, 
$\varphi\in W^{1,p}(\Omega)$, $p\in(1,\infty)$. 
If $u$ is a solution of 
\[
\begin{cases}
-\Delta u=f,\quad &\text{in }\Omega,\\
\dfrac{\partial u}{\partial \nu}=\varphi, &\text{on }\partial\Omega,
\end{cases}
\]
then 
\[
\|u\|_{W^{2,p}(\Omega)}
\leq C\big(\|f\|_{L^p(\Omega)}+\|\varphi\|_{W^{1,p}(\Omega)}\big).
\]
\end{lemma}

\begin{proof}
This is a classical result.
\end{proof}

For each $p\geq 0$, suppose that $u_\beta$ is a positive solution to equation \eqref{eq0.1}. Then we have 

\begin{lemma}\label{Lemma2.1}
For each $\beta>0$,  we have that there exists constant $c$, which depends on $\beta$,  such that 
$$
u_\beta\ge c>0.
$$
\end{lemma}

\begin{proof}
Since 
$$
-\Delta u_\beta=u_\beta^p> 0,
$$
then by strong maximum principle, we can get that there exists $x_0\in\partial\Omega$ such that 
$$
u_\beta(x_0)=\min_{\bar{\Omega}}u_\beta<u(x),\quad x\in \Omega.
$$
Thus, by Hopf's Lemma, we have 
$$
\frac{\partial u_\beta}{\partial \nu}(x_0)<0.
$$
By the Robin boundary condition, we derive 
$$
u_\beta(x_0)=-\frac{1}{\beta}\frac{\partial u_\beta}{\partial \nu}(x_0)>0.
$$
Therefore, we prove the claim. 
\end{proof}

\begin{lemma}\label{lemma2.2}
If $p=1$, let $\lambda_{1,\beta}$ denote the first eigenvalue of the Robin problem
\[
\begin{cases}
-\Delta \phi=\lambda_{1,\beta}\phi, &\text{in }\Omega,\\
\dfrac{\partial \phi}{\partial \nu}+\beta \phi=0, &\text{on }\partial\Omega.
\end{cases}
\]
Then we have
\begin{equation}\label{2.7}
\lambda_{1,\beta}\le\beta\frac{|\partial\Omega|}{|\Omega|}\to0,\quad\hbox{as }\beta\to0.
\end{equation}
\end{lemma}

\begin{proof}
We first  recall the definition of $\lambda_{1,\beta}$, that is
$$
\lambda_{1,\beta}=\inf_{\phi\in H^1(\Omega)\backslash\{0\}}\frac{\int_{\Omega}|\nabla \phi|^2dx+\beta\int_{\partial\Omega}\phi^2dS}{\int_{\Omega}\phi^2dx}.
$$
Then, we take $\phi=1$, we get 
$$
 \lambda_{1,\beta}\leq\beta\frac{|\partial\Omega|}{|\Omega|}
$$
\end{proof}

\section{Existence results}
In this section, we shall prove some existence results for problem \eqref{eq0.1}. Some of them are well-known, as the next one.
\begin{theorem}
Suppose that $1<p<2^*-1$, then equation \eqref{eq0.1} admits a positive  solution for each $\beta>0$.
\end{theorem}

\begin{proof}
The proof is standard. For the convenience of readers, we provide a brief proof here. It is sufficient to prove that 
$$
I_\beta=\inf_{u\in \mathcal{M}}\Big(\int_{\Omega} |\nabla u|^2dx+\beta\int_{\partial\Omega}u^2dS\Big)
$$
can be achieved, where $\mathcal{M}$ is defined by 
$$
\mathcal{M}=\Big\{ u\in H^1(\Omega) : \int_{\Omega}|u|^{p+1}dx=1\Big\}.
$$
Obviously, one can get $I_\beta>0$. Let $\{u_n\}$ be a minimizing sequence of $I_\beta$. Then we have 
$$
\int_{\Omega} |\nabla u_n|^2dx+\beta\int_{\partial\Omega}u_n^2dS=I_\beta+o_n(1),
$$
and
$$
\int_{\Omega}|u_n|^{p+1}dx=1.
$$
Note that the norm 
$$
\|u\|_\beta:=\Big(\int_{\Omega} |\nabla u|^2dx+\beta\int_{\partial\Omega}u^2dS\Big)^{\frac{1}{2}}
$$
is equivalent  to the $H^1(\Omega)$-norm, then we can get that  the sequence $\{u_n\}$ is bounded in $H^1(\Omega)$. Moreover, by the compactness of the Sobolev embeddings we deduce that there exists $u_0\in H^1(\Omega)$ such that  
$$
\int_{\Omega} |\nabla u_0|^2dx+\beta\int_{\partial\Omega}u_0^2dS=I_\beta,\quad \int_{\Omega}|u_0|^{p+1}dx=1.
$$
Thus we prove that $I_\beta$ can be achieved. 

Finally, we may assume that $u_0$ is nonnegative, otherwise we may choose $|u_0|$ replace by $u_0$.  After a suitable scaling, we can obtain a nonnegative solution of \eqref{eq0.1}. By strong maximum principle, we can derive a positive solution  of \eqref{eq0.1}.
\end{proof}

Define the Nehari manifold 
$$
\mathcal{N}=\Big\{ u\in H^1(\Omega)\backslash\{0\} : \int_{\Omega}|\nabla u|^2dx+\beta\int_{\partial\Omega}u^2dS=\int_{\Omega}|u_+|^{p+1}dx\Big\},
$$
and 
$$
m_\beta=\inf_{u\in\mathcal{N}} J_\beta(u).
$$
Then we have the following conclusion.

\begin{theorem}
Suppose that $p=2^*-1$. If 
$$
0<\beta<2^{-\frac{2}{N}}|\partial\Omega|^{-1}|\Omega|^{\frac{2}{p+1}}\mathcal{S},
$$ 
then equation \eqref{eq0.1} has a ground state solution.
\end{theorem}

\begin{proof}
In order to prove the existence of ground state solutions, it is sufficient to prove that $m_\beta$ 
can be achieved. Taking 
$$
u=\Big(\frac{\beta|\partial\Omega|}{|\Omega|}\Big)^{\frac{1}{p-1}}\in \mathcal{N},
$$
we get 
$$
m_\beta\le J_\beta\Big(\frac{\beta|\partial\Omega|}{|\Omega|}\Big)^{\frac{1}{p-1}}=\frac{1}{N}\beta^{\frac{p+1}{p-1}}|\partial\Omega|^{\frac{p+1}{p-1}}|\Omega|^{-\frac{2}{p-1}}<\frac{1}{2N}\mathcal{S}^{\frac{N}{2}}.
$$
Let $\{u_n\}\subset\mathcal{N}$ be a minimizing sequence. Then by Ekeland's variational principle, we have that there exists a sequence $\{v_n\}$ such that 
$$
J_\beta(v_n)\to c,\quad J_\beta|_{\mathcal{N}}^\prime(v_n)\to0,\quad \|v_n-u_n\|\to0,\quad \text{as}\ n\to\infty.
$$
Since 
$$
m_\beta<\frac{1}{2N}\mathcal{S}^{\frac{N}{2}}, 
$$
then $J_\beta$ satisfies $(PS)$ condition. Thus, we can deduce that there exists $u_\beta\in\mathcal{N}$ such that $J_\beta(u_\beta)=m_\beta$. Therefore, we prove the existence of ground state solutions.
\end{proof}

We first observe that for $p>\frac{N+2}{N-2}$ there are no known existence results in the literature. For this reason, we begin by providing an existence result in the case of the ball. The existence of solutions in the supercritical case for small $\beta$ in general domains remains an interesting open problem.

As recalled earlier, in the supercritical case it is more difficult to obtain existence results for solutions. However, in the case of the ball, the existence of {\em slow decay solutions} for the problem in $\R^n$ allows one to deduce the existence of a solution in the case of the ball.

\begin{theorem}
Suppose that $p>\frac{N+2}{N-2}$, $0<\beta<\frac{2}{p-1}$ and $\Omega$ is a ball.  Then there exists a radial solution of equation \eqref{eq0.1}.
\end{theorem}

\begin{proof}
We recall a key result obtained in \cite{SZ} (see also \cite{RS}) where it was proved that the following problem
\begin{equation}
\begin{cases}
-\Delta u=u^p,\quad &\text{in}\ \R^N,\\
u>0, &\text{in}\ \R^N,
\end{cases}
\end{equation}
has the unique radial solution $U$ and 
\begin{equation}\label{li-2.10}
\lim_{|x|\to\infty}|x|^{a}U(|x|)=b^{\frac{1}{p-1}},
\end{equation}
where $a=\frac{2}{p-1}$ and $b=a(N-2-a)$. The similar result can be also founded in \cite{SZ}. 

From \cite{RS}, one can prove that 
\begin{equation}\label{li-2.11}
\lim_{|x|\to\infty}|x|^{a+1}U^\prime(|x|)=-ab^{\frac{1}{p-1}}.
\end{equation}
Indeed, let $s=\log r$ and $v(s)=r^aU(r)$, then we have $U(r)=e^{-s}v(s)$. Compute 
\begin{equation}\label{li-2.12}
r^{a+1}U^\prime(r)=v^\prime(s)-av(s).
\end{equation}
Rupflin and Struwe \cite{RS} showed that 
$$
\lim_{s\to\infty}v^\prime(s)=0,\quad \lim_{s\to\infty}v(s)=b^{\frac{1}{p-1}}.
$$
Thus, we can get 
$$
\lim_{s\to\infty}\Big( v^\prime(s)-av(s)\Big)=-ab^{\frac{1}{p-1}}.
$$
So by \eqref{li-2.12}, we can get $r^{a+1}U^\prime(r)\to -ab^{\frac{1}{p-1}}$ as $r\to\infty$.

Without loss of generality, we may suppose that $\Omega=B_{1}(0)$. Define 
$$
U_\delta(r)=\delta^{\frac{2}{p-1}}U(\delta r),
$$
then it is immediate that $U_\delta$ verifies equation in \eqref{eq0.1} for all $\delta>0$. Now, we need to find a suitable $\delta$ such that the boundary condition in \eqref{eq0.1}  holds.  Since $U^\prime(0)=0$, we have 
\begin{equation*}
U_\delta^\prime(1)+\beta U_\delta(1)=\delta^{\frac{2}{p-1}}(\delta U^\prime(\delta)+\beta U(\delta))=\delta^{\frac{2}{p-1}}(\beta U(0)+o(1)),\quad \text{as}\ \delta\to 0.
\end{equation*}
By \eqref{li-2.10} and \eqref{li-2.11}, we have 
\begin{equation*}
U_\delta^\prime(1)+\beta U_\delta(1)=\delta^{\frac{2}{p-1}}(\delta U^\prime(\delta)+\beta U(\delta))=b\Big(\beta-\frac{2}{p-1}+o(1) \Big),\quad \text{as}\ \delta\to \infty.
\end{equation*}
Since $0<\beta<\frac{2}{p-1}$, then we have 
$$
\lim_{\delta\to 0}\Big(U_\delta^\prime(1)+\beta U_\delta(1)\Big)>0,\quad \lim_{\delta\to \infty}\Big(U_\delta^\prime(1)+\beta U_\delta(1)\Big)<0.
$$
Thus, there exists $\hat{\delta}=\hat{\delta}(\beta)>0$ such that 
$$
U_{\hat{\delta}}^\prime(1)+\beta U_{\hat{\delta}}(1)=0.
$$
So $U_{\hat{\delta}}$ solves equation \eqref{eq0.1}.
\end{proof}

\section{Uniqueness results}

In this section, our purpose is to prove some uniqueness results. We begin with the sublinear case.

\subsection{The case $0\le p<1$}  

\begin{theorem}\label{Theorem2.4}
If $0\le p<1$, then the problem \eqref{eq0.1} has exactly one positive solution for each $\beta>0$.
\end{theorem}

\begin{proof}
The proof is inspired by the work of Br\'ezis-Oswald \cite{BO}. Suppose that $u$ and $v$ are two positive solutions of problem \eqref{eq0.1}.
A direct computation yields that 
$$
\nabla \Big(\frac{u^2}{v}\Big)=\frac{2uv\nabla u-u^2\nabla v}{v^2},\quad \nabla \Big(\frac{v^2}{u}\Big)=\frac{2vu\nabla v-v^2\nabla u}{u^2}.
$$
It follows from Lemma \ref{Lemma2.1} that 
$\frac{u^2}{v}$ and $\frac{v^2}{u}$ belong to $H^1(\Omega)$. Multiplying equation \eqref{eq0.1} by $u-\frac{v^2}{u}$, we can get 
\begin{align}\label{2.4}
&\int_{\Omega}u^{p-1}(u^2-v^2)dx\nonumber\\
&=\int_{\Omega}u^{p}\Big(u-\frac{v^2}{u}\Big)dx \nonumber\\
&=\int_{\Omega}\nabla u\nabla \Big(u-\frac{v^2}{u} \Big)dx-\int_{\partial\Omega}\frac{\partial u}{\partial \nu}\Big(u-\frac{v^2}{u} \Big)dS\nonumber\\
&=\int_{\Omega}\Big(|\nabla u|^2-\frac{2v}{u}\nabla u\nabla v+\frac{v^2}{u^2}|\nabla u|^2 \Big)dx+\beta\int_{\partial\Omega}(u^2-v^2)dS.
\end{align}
Similarly, we can derive 
\begin{align}\label{2.5}
&\int_{\Omega}v^{p-1}(v^2-u^2)dx\nonumber\\
&=\int_{\Omega}\Big(|\nabla v|^2-\frac{2u}{v}\nabla v\nabla u+\frac{u^2}{v^2}|\nabla v|^2 \Big)dx+\beta\int_{\partial\Omega}(v^2-u^2)dS.
\end{align}
Thus, by \eqref{2.4} and \eqref{2.5} we derive 
\begin{equation*}
\int_{\Omega}(u^{p-1}-v^{p-1})(u^2-v^2)dx=\int_{\Omega}\Big(\Big|\nabla u-\frac{u}{v}\nabla v|^2+\Big|\nabla v-\frac{v}{u}\nabla u\Big|^2\Big)dx\ge 0.
\end{equation*}
Since the function $s^{p-1}$ is strictly decreasing in $(0,\infty)$, then we must have $u=v$.
\end{proof}

\subsection {The case $p=1$} 

\begin{lemma}\label{lem4.2}
Suppose that $p=1$, then $\lambda_{1,\beta}$ is simple.  Moreover,   for each $\beta>0$ the eigenvalue problem \eqref{eig} has exactly one positive solution with $||u_\beta||_{L^\infty(\Omega)}=1$.
\end{lemma}

\begin{proof}
The Robin eigenvalue problem is a known result.
\end{proof}


For superlinear case, in order to prove the uniqueness results, we first establish  the following conclusion.

\begin{theorem}\label{lemma2.5}
Suppose that $p>1$. If there exists a constant $C>0$, which is independent of $\beta$, such that 
\begin{equation}\label{l2.4}
u_\beta\le C,
\end{equation}
then we have 
$$
u_\beta\to 0\quad  \text{in} \ C^2_{loc}(\bar{\Omega})\  \text{as}\ \beta\to0.
$$
Moreover, we have that $u_\beta$ is the unique bounded  solution of equation \eqref{eq0.1} for $\beta$ small enough.
\end{theorem}

\begin{proof}
Since $u_\beta$ is bounded, then by Lemma \ref{lemma2.3} we have that $u_\beta\in W^{2,q}(\Omega)$ for all $q<\infty$. Thus, by elliptic equation regularity theory, we can get there exists $u_0$ such that 
$$
u_\beta\to u_0 \quad \text{in}\ C^2_{loc}(\bar{\Omega}),
$$
and $u_0$ solves the equation 
\begin{equation}
\begin{cases}
-\Delta u_0=u_0^p,\quad  &\text{in}\ \Omega,\\
\frac{\partial u_0}{\partial \nu}=0, &\text{on}\ \partial\Omega.
\end{cases}
\end{equation}
Note that $u_0\ge 0$ and 
$$
\int_{\Omega}u_0^pdx=\int_{\Omega}-\Delta u_0dx=-\int_{\partial\Omega}\frac{\partial u_0}{\partial \nu}dS=0,
$$
Therefore, we have $u_0=0$.

In the following, we prove the uniqueness of bounded solutions equation \eqref{eq0.1}.
Suppose that $u$ and $v$ are two distinct positive solution. Let $w=u-v$, then $w$ is a solution of 
\begin{equation*}
\begin{cases}
-\Delta w=a(x)w,\quad &\text{in}\ \Omega,\\
\frac{\partial w}{\partial \nu}+\beta w=0, &\text{on}\ \partial\Omega,
\end{cases}
\end{equation*}
where $a(x)$ is defined by 
$$
a(x)=p(\theta u+(1-\theta)v)^{p-1}, \quad \theta\in(0,1).
$$
Note that $w\neq 0$, then $\|w\|_\infty>0$, where $\|\cdot\|_\infty$ is defined by $L^\infty(\Omega)$ norm. Define 
$$
\tilde{w}=\frac{w}{\|w\|_\infty},
$$
then we get 
\begin{equation*}
\begin{cases}
-\Delta \tilde{w}=a(x)\tilde{w},\quad &\text{in}\ \Omega,\\
\frac{\partial \tilde{w}}{\partial \nu}+\beta \tilde{w}=0, &\text{on}\ \partial\Omega,
\end{cases}
\end{equation*}
Since  $|\tilde{w}|\le 1$ and $a(x)\to 0$ uniformly in $\Omega$,  by elliptic equation regularity, we can conclude that $\tilde{w}\to \tilde{w}_0$ in $C_{loc}^2(\bar{\Omega})$ and $\tilde{w}_0$ is a solution of 
\begin{equation*}
\begin{cases}
-\Delta \tilde{w}_0=0,\quad &\text{in}\ \Omega,\\
\frac{\partial \tilde{w}_0}{\partial \nu}=0, &\text{on}\ \partial\Omega.
\end{cases}
\end{equation*}
This implies that $\tilde{w}_0=C$ for some constant $C$. We observe that $\|\tilde{w}\|_\infty=1$, so we get $C=1$ or $C=-1$. 
Without loss of generality, we may suppose that $C=1$, then for $\beta$ small enough we get $w>0$. 
Direct computation yields that 
$$
\int_{\Omega}\nabla u\nabla vdx+\beta\int_{\partial\Omega}uvdS=\int_{\Omega}u^pvdx, 
$$
and
$$
\int_{\Omega}\nabla v\nabla udx+\beta\int_{\partial\Omega}vudS=\int_{\Omega}v^pudx, 
$$
Thus, we have 
\begin{equation}\label{cgl3.1}
\int_{\Omega}u^pv-v^pudx=0.
\end{equation}
It follows from \eqref{cgl3.1} that 
\begin{equation*}
0=\int_{\Omega}u^pv-v^pudx=\int_{\Omega}uv(u^{p-1}-v^{p-1})dx=\int_{\Omega}b(x)uvwdx, 
\end{equation*}
where $b(x)$ is defined by 
$$
b(x)=(p-1)(\vartheta u+(1-\vartheta)v)^{p-2}, \quad \vartheta\in(0,1).
$$
We observe that  $b(x)>0$ since $p>1$.  By the facts that $u$, $v$ and $w$ are positive, we have 
\begin{equation*}
\int_{\Omega}b(x)uvwdx>0.
\end{equation*}
This is a contradiction which ends the proof.
\end{proof}

As a direct consequence Theorem \ref{lemma2.5}, it is sufficient to check that condition \eqref{l2.4} is satisfied for subcritical case, critical case and supercritical case respectively.

\subsection{The case $1<p<2^*-1$}

\begin{lemma}\label{lemma2.7}
Suppose that $1<p<2^*-1$. Then  the condition \eqref{l2.4} is satisfied for $\beta$ sufficiently small.
\end{lemma}

\begin{proof}
The proof is standard, one can check it by scaling arguments in \cite{GS1}. We only sketch the proof.  If Lemma \ref{lemma2.7} is false,  then we may suppose that $\|u_\beta\|_\infty\to \infty$ as $\beta\to 0$.  Let $x_\beta\in \Omega$ be such that 
$$
u_\beta(x_\beta)=\|u_\beta\|_{\infty}.
$$
Define
\begin{equation*}
v_\beta(y)=\frac{1}{\|u_\beta\|_\infty}u_\beta(x),\quad x=\frac{y}{\|u_\beta\|_\infty^{\frac{p-1}{2}}}+x_\beta\in\Omega.
\end{equation*}
then we get 
\begin{equation*}
\begin{cases}
-\Delta v_\beta=v_\beta^p,\quad &\text{in}\ \Omega_\beta,\\
\nabla v_\beta\cdot \tilde{\nu}_y+\beta v_\beta=0, &\text{on}\ \partial\Omega_\beta,
\end{cases}
\end{equation*}
where $\Omega_\beta=\|u_\beta\|_\infty^{\frac{p-1}{2}}(\Omega-x_\beta)$ and $\tilde{\nu}_y$ is defined by  
$$
\tilde{\nu}_y=\|u_\beta\|_\infty^{\frac{p-1}{2}}\nu_x,\qquad x=\frac{y}{\|u_\beta\|_\infty^{\frac{p-1}{2}}}+x_\beta\in\partial\Omega.
$$
In the following, we will distinguish two cases:

(i) If $\|u_\beta\|^{\frac{p-1}{2}}_\infty dist(x_\beta,\partial\Omega)\to \infty$ as $\beta\to 0$, then $\Omega_\beta\to \R^N$;

(ii) If $\|u_\beta\|^{\frac{p-1}{2}}_\infty dist(x_\beta,\partial\Omega)\to C<\infty$ as $\beta\to 0$, then $\Omega_\beta\to \R_+^N$.

If (i) occurs. By the fact that $\|v_\beta\|_\infty=1$,  we have that there exists $v_0$ such that 
$$
v_\beta\to v_0,\quad \text{in}\ C^2_{loc}(\R^N),
$$
where $v_0$ solves the equation
\begin{equation}\label{eq3.3}
\begin{cases}
-\Delta v_0=v_0^p,\quad \text{in}\ \R^N,\\
v_0(0)=1.
\end{cases}
\end{equation}
Note that equation \eqref{eq3.3} has no positive solution since $1<p<2^*-1$. This is a contradiction. Thus, we prove that case (i) can't occur.

If (ii) occurs, then $x_\beta\to x_0\in\partial\Omega$. By straightening $\partial\Omega$ in a neighborhood of $x_0$, we may suppose that near $x_0$, $\partial\Omega$ is contained in hyperplane $x_N=0$. Arguing as before, we can get that there exists $v_0$ such that $v_0$ is a positive solution of 
\begin{equation}\label{e3.8}
\begin{cases}
-\Delta v_0=v_0^p,\quad &\text{in}\ \R^N_+,\\
\frac{\partial v_0}{\partial y_N}=0, &\text{on}\ y_N=0.
\end{cases}
\end{equation}
Making a even extension with respect to the hyperplane $y_N=0$, that is 
\begin{equation*}
\tilde{v}_0(y)=
\begin{cases}
v_0(y_1,\cdots,y_{N-1},y_N),\quad &\text{if}\ y_N\ge0,\\
v_0(y_1,\cdots,y_{N-1},-y_N),\quad &\text{if}\ y_N<0.
\end{cases}
\end{equation*}
Thus, we can get that  $\tilde{v}_0$ is a solution of equation \eqref{eq3.3}. 
This is also a contradiction. Hence, we prove that case (ii) can't occur.
\end{proof}

\subsection{The case $p=2^*-1$} 
In this case, we mainly study the uniqueness of ground state solutions.

\begin{lemma}\label{lemma2.8}
If $u_\beta$ is a ground state solution, then $u_\beta$ satisfies the condition \eqref{l2.4} for $\beta$ sufficiently small.
\end{lemma}

\begin{proof}
In the following, we will prove that $u_\beta\le C$ for $\beta$ sufficient small. If not, we may suppose that $\|u_\beta\|_\infty\to \infty$ as $\beta\to 0$. Arguing as Lemma \ref{lemma2.7}, then we can deduce that $v_0$ solves equation \eqref{eq3.3} or \eqref{e3.8}. 
If $v_0$ solves equation \eqref{eq3.3}, then $v_0$ must have the form 
$$
v_0(y)=[N(N-2)]^{\frac{N-2}{4}}\mu^{\frac{N-2}{2}}\Big(1+\mu^2|y|^2\Big)^{-\frac{N-2}{2}}
$$
for some $\mu>0$. Since $v_0(0)=1$, then we have 
$$
\mu=[N(N-2)]^{-\frac{1}{2}}. 
$$
We observe that 
$$
\lim_{\beta\to0}J_\beta(u_\beta)=\lim_{\beta\to0}\frac{1}{N}\int_{\Omega}u_\beta^{p+1}dx=\lim_{\beta\to0}\frac{1}{N}\int_{\Omega_\beta}v_\beta^{p+1}dy\ge \frac{1}{N}\int_{\R^N}v_0^{p+1}dy=\frac{1}{N}\mathcal{S}^{\frac{N}{2}},
$$
which is a contradiction since 
$$
\lim_{\beta\to0}J_\beta(u_\beta)=\lim_{\beta\to 0}m_\beta=0.
$$
On the other hand, if $v_0$ solves equation \eqref{e3.8}, then by even extension with respect to $y_N$ we can 
get a positive solution $\tilde{v}_0$ of equation \eqref{eq3.3}. Hence, we can get 
\begin{equation*}
\lim_{\beta\to0}J_\beta(u_\beta)\ge \frac{1}{N}\int_{\R_+^N}v_0^{p+1}dy=\frac{1}{2N}\int_{\R^N}\tilde{v}_0^{p+1}dy=\frac{1}{2N}\mathcal{S}^{\frac{N}{2}}.
\end{equation*}
This is also a contradiction. So we prove that $u_\beta$ is bounded.
\end{proof}

\subsection{The case $p>2^*-1$} In this case, we mainly consider the uniqueness of positive radial solutions.

\begin{lemma}\label{lemma2.6}
Suppose that $p>1$ and $\Omega$ is a ball. If $u_\beta$ is a radial solution, then 
$\|u_\beta\|_{\infty}\to 0$  as $\beta\to 0$. Moreover, the condition \eqref{l2.4} is satisfied for $\beta$ sufficiently small.
\end{lemma}

\begin{proof}
Without loss of generality, we may suppose that $\Omega=B_{1}(0)$. By strong maximum principle, we get 
\begin{equation*}
u_\beta(r)>u_\beta(1),\quad \forall r<1,
\end{equation*}
where $r=|x|$. Since 
\begin{equation}
\beta\int_{\partial B_1(0)}u_\beta\, dS
=-\int_{\partial B_1(0)}\frac{\partial u_\beta}{\partial\nu}dS
=\int_{B_1(0)}-\Delta u_\beta\, dx
=\int_{B_1(0)}u_\beta^p\, dx,
\end{equation}
then we have 
$$
\beta |\partial B_1(0)|u_\beta(1)=\int_{B_1(0)}u_\beta^pdx>[u_\beta(1)]^p| B_1(0)|.
$$
This implies that 
$$
u_\beta(1)<\Big( \frac{\beta |\partial B_1(0)|}{| B_1(0)|}\Big)^{\frac{1}{p-1}}=N^{\frac{1}{p-1}}\beta^{\frac{1}{p-1}}.
$$
By \eqref{2.7}, we get 
\begin{equation}\label{2.8}
\int_{B_1(0)}u_\beta^pdx=\beta |\partial B_1(0)|u_\beta(1)<N^{\frac{1}{p-1}}\beta^{\frac{p}{p-1}}|\partial B_1(0)|.
\end{equation}

In the following, we will prove that $\|u_\beta\|_\infty\to 0$. We argue by contradiction, if there exists a constant $C>0$ such that $\|u_\beta\|_\infty\geq C>0$ as $\beta\to 0$.
Let
$$
\tilde{u}_\beta=\frac{u_\beta}{\|u_\beta\|_\infty};
$$
then $\tilde{u}_\beta$ solves the following equation:
\begin{equation}
\begin{cases}
-\Delta \tilde{u}_\beta=u_\beta^{p-1}\tilde{u}_\beta,\quad &\text{in}\ B_1(0),\\
\frac{\partial \tilde{u}_\beta}{\partial \nu}+\beta \tilde{u}_\beta=0, &\text{on}\ \partial B_1(0).
\end{cases}
\end{equation}
A direct computation yields 
\begin{equation}\label{li-2.19}
\int_{B_1(0)}|\nabla \tilde{u}_\beta|^2dx+\beta\int_{\partial B_1(0)}\tilde{u}_\beta^2dS=\int_{B_1(0)}u_\beta^{p-1}\tilde{u}_\beta^2dx.
\end{equation}
Since $0<\tilde{u}_\beta\leq 1$, by  \eqref{2.8}, we can deduce that 
\begin{equation}\label{li-2.20}
\int_{B_1(0)}u_\beta^{p-1}\tilde{u}_\beta^2dx\leq \int_{B_1(0)}u_\beta^{p-1}dx\leq |B_1(0)|^{\frac{1}{p}}\Big(\int_{B_1(0)}u_\beta^{p}dx\Big)^{\frac{p-1}{p}}\leq |\partial B_1(0)|\beta.
\end{equation}
Combining \eqref{li-2.19} with \eqref{li-2.20}, we obtain
\begin{equation}\label{li-2.21}
\int_{B_1(0)}|\nabla \tilde{u}_\beta|^2dx\to 0,\quad \text{as}\ \beta\to 0.
\end{equation}
On the other hand, we observe that 
$$
\int_{B_1(0)}u_\beta^{p-1}\tilde{u}_\beta^2dx=\|u_\beta\|_\infty^{p-1}\int_{B_1(0)}\tilde{u}_\beta^{p+1}dx\geq C^{p-1}\int_{B_1(0)}\tilde{u}_\beta^{p+1}dx.
$$
Combining this with \eqref{li-2.20}, we obtain $\tilde{u}_\beta\to 0$ in $L^{p+1}(B_1(0))$ as $\beta\to 0$. Since $p>1$, we derive that 
\begin{equation}\label{li-2.22}
\int_{B_1(0)}\tilde{u}_\beta^2dx\to 0, \quad\text{as}\  \beta\to 0. 
\end{equation}
It follows from \eqref{li-2.21} and \eqref{li-2.22} that $\tilde{u}_\beta\to 0$ in $H^{1}(B_1(0))$ as $\beta\to 0$. Moreover, we have $\tilde{u}_\beta\to 0$ as $\beta\to 0$. Hence, $\tilde{u}_\beta=o(1)$ as $\beta\to 0$, which contradicts $\|\tilde{u}_\beta\|_\infty=1$.
\end{proof}

{\bf Proof of Theorem \ref{Theorem0.1}:}  Firstly, by Theorem \ref{Theorem2.4}, we can know that claim (1) is true. Secondly,  Lemma \ref{lem4.2} implies that claim (2) is also true. Then for $p>1$, as a direct consequence of Theorem \ref{lemma2.5}, it is sufficient to examine the cases in which this condition \eqref{l2.4} is actually satisfied. 
As a direct consequence of Lemma \ref{lemma2.7}-\ref{lemma2.6}, we can prove  \eqref{l2.4}. Hence, claim (3)-(5) are true.

\section{Asymptotic behavior of the solution}
In this section, we study the asymptotic behavior of the solution to equation \eqref{eq0.1}. We begin with the case $0\le p<1$.

\subsection{The case $0\le p<1$}

In order to give an idea of the general phenomenon, let us briefly consider the case of the ball, i.e.$\Omega=B_1(0)$. 
 In this case, by Theorem \ref{Theorem0.1},  positive solutions to  \eqref{eq0.1} are  radial. 
Moreover, by means of maximum principle, we get 
$$
u(r)=u(|x|)> u(1),\quad \forall r<1.
$$
It follows from \eqref{2.7} that 
$$
\beta u(1)=\int_{0}^1u^pr^{N-1}dr> \frac{[u(1)]^p}{N},
$$
which implies that 
$$
u(1)> N^{-\frac{1}{1-p}}\beta^{-\frac{1}{1-p}}\to \infty,\quad \text{as}\ \beta\to 0.
$$
Thus, we observe that radial solutions to equation \eqref{eq0.1} blow up in the ball. Naturally, for a general bounded domain $\Omega$, one might ask:\\
{\em do all positive solutions to equation \eqref{eq0.1} still blow up as $\beta\to 0$?}\\
Moreover, what is their exact asymptotic behavior as $\beta\to 0$? In order to answer these questions,  
we first prove an estimate for $u_\beta$ in general domains.

\begin{lemma}\label{Lemma3.1}
Suppose that $0\leq p<1$; then it holds that 
$$
\|u_\beta\|_\infty\ge \beta^{-\frac{1}{1-p}}|\partial\Omega|^{-\frac{1}{1-p}}|\Omega|^{\frac{1}{1-p}}.
$$
\end{lemma}

\begin{proof}
Let $\phi>0$ be the first eigenfunction of Robin problem 
\begin{equation}\label{eq0.9}
\begin{cases}
-\Delta \phi=\lambda_{1,\beta}\phi,\quad &\text{in}\ \Omega,\\
\frac{\partial \phi}{\partial\nu}+\beta\phi=0 &\text{on}\ \partial\Omega.
\end{cases}
\end{equation}
Multiplying \eqref{eq0.9} by $u_\beta$ and integrating by parts, we get 
\begin{equation}\label{eq0.10}
\int_{\Omega}\nabla \phi\nabla u_\beta dx+\beta\int_{\partial\Omega}\phi u_\beta dS=\lambda_{1,\beta}\int_{\Omega}\phi u_\beta dx.
\end{equation}
On the other hand, multiplying equation \eqref{eq0.1} by $\phi$ and integrating by parts, we have 
\begin{equation}\label{eq0.11}
\int_{\Omega}\nabla u_\beta\nabla \phi dx+\beta\int_{\partial\Omega} u_\beta\phi dS=\int_{\Omega}u_\beta^p\phi dx.
\end{equation}
It follows from \eqref{eq0.10} and \eqref{eq0.11} that 
\begin{equation}\label{eq0.12}
\int_{\Omega}\phi u_\beta^p(1-\lambda_{1,\beta}u_\beta^{1-p})dx=0.
\end{equation}
Thus, we claim that 
\begin{equation}\label{eq3.5}
\min_{\Omega}(1-\lambda_{1,\beta}u_\beta^{1-p})\leq 0.
\end{equation}
If not, we may suppose that 
\begin{equation*}
\min_{\Omega}(1-\lambda_{1,\beta}u_\beta^{1-p})> 0,
\end{equation*}
which implies that 
$$
\int_{\Omega}\phi u_\beta^p(1-\lambda_{1,\beta}u_\beta^{1-p})dx\ge \min_{\Omega}(1-\lambda_{1,\beta}u_\beta^{1-p})\int_{\Omega}\phi u_\beta^pdx>0.
$$
This is a contradiction with \eqref{eq0.12}. So we prove \eqref{eq3.5}.

Note that \eqref{eq3.5} is equivalent to 
$$
\|u_\beta\|_\infty\ge \lambda_{1,\beta}^{-\frac{1}{1-p}}.
$$
Then by Lemma \ref{lemma2.2}, we have 
$$
\lambda_{1,\beta}\leq \beta\frac{|\partial\Omega|}{|\Omega|}.
$$
So we can get 
$$
\|u_\beta\|_\infty\ge \beta^{-\frac{1}{1-p}}|\partial\Omega|^{-\frac{1}{1-p}}|\Omega|^{\frac{1}{1-p}}.
$$
\end{proof}

{\bf Proof of Theorem \ref{Theorem1.2}:}
We first recall the definition of $\tilde{u}_\beta$, namely 
$$
\tilde{u}_\beta=\frac{u_\beta}{\|u_\beta\|_\infty}.
$$
Note that $\tilde{u}_\beta$ solves the equation 
\begin{equation}\label{0.1}
\begin{cases}
-\Delta \tilde{u}_\beta=\|u_\beta\|_{\infty}^{p-1}\tilde{u}_\beta^p,\quad  &\text{in}\ \Omega,\\
\frac{\partial \tilde{u}_\beta}{\partial \nu}+\beta \tilde{u}_\beta=0, &\text{on}\ \partial\Omega.
\end{cases}
\end{equation}
It follows from Lemma \ref{Lemma3.1} that 
$$
\|u_\beta\|_\infty\to \infty,\quad \text{as}\ \beta\to 0.
$$
Then $\|u_\beta\|_\infty^{p-1}\to 0$ as $\beta\to 0$ since $p<1$.
Thus by the fact that $0<\tilde{u}_\beta\le 1$, we can get that 
$$
\tilde{u}_\beta\to \tilde{u}_0,\quad  \text{in} \ C^2_{loc}(\bar{\Omega})
$$ 
and $\tilde{u}_0$ is a solution of 
\begin{equation}\label{0.2}
\begin{cases}
-\Delta \tilde{u}_0=0,\quad  &\text{in}\ \Omega,\\
\frac{\partial \tilde{u}_0}{\partial \nu}=0, &\text{on}\ \partial\Omega,
\end{cases}
\end{equation}
which implies that $\tilde{u}_0\equiv C$ for some constant. Since $0<\tilde{u}_\beta\le 1$, we have $0\le C\le 1$. If $C<1$, then we can get $\sup \tilde{u}_\beta<1$ for $\beta$ small enough. This is a contradiction with the fact that $\|\tilde{u}_\beta\|_\infty\equiv 1$ for any $\beta>0$. Therefore, we prove that $C=1$.

We observe that 
\begin{equation}
\beta\int_{\Omega}u_\beta dS=\int_{\partial\Omega}-\frac{\partial u_\beta}{\partial \nu}dS=\int_{\Omega}-\Delta u_\beta dx=\int_{\Omega}u_\beta^pdx,
\end{equation}
then we have 
\begin{equation}\label{0.4}
\beta\|u_\beta\|_{\infty}\int_{\partial\Omega}\tilde{u}_\beta dS=\|u_\beta\|_{\infty}^p\int_{\Omega}\tilde{u}_\beta^pdx.
\end{equation}
Since $\tilde{u}_\beta\to 1$ in $C^2_{loc}(\bar{\Omega})$, then by \eqref{0.4} we can have 
\begin{equation}\label{0.5}
\beta\|u_\beta\|_{\infty}(|\partial\Omega|+o(1))=\|u_\beta\|_{\infty}^p(|\Omega|+o(1)).
\end{equation}
Therefore, we can derive 
\begin{equation}\label{0.6}
\|u_\beta\|_{\infty}=\beta^{-\frac{1}{1-p}}\Big(|\partial\Omega|^{-\frac{1}{1-p}}|\Omega|^{\frac{1}{1-p}} +o(1)\Big)
\end{equation}
and 
\begin{equation}\label{0.7}
u_\beta=\|u_\beta\|_{\infty}\tilde{u}_\beta=\beta^{-\frac{1}{1-p}}\Big(|\partial\Omega|^{-\frac{1}{1-p}}|\Omega|^{\frac{1}{1-p}} +o(1)\Big).
\end{equation}
\qed

\subsection{The case $p=1$}
In this section,  based on Lemma \ref{lemma2.2}, we will prove the asymptotic behavior of the first eigenfunction.

{\bf Proof of Theorem \ref{Theorem1.3}:} 
Multiplying equation \eqref{1.5} by $u_\beta$, we get 
$$
\int_{\Omega}|\nabla u_\beta|^2dx+\beta\int_{\Omega}u_\beta^2dS=\lambda_{1,\beta}\int_{\Omega}u_\beta^2dx,
$$
which implies that 
$$
\int_{\Omega}|\nabla u_\beta|^2dx\leq \lambda_{1,\beta}\int_{\Omega}u_\beta^2dx\leq \lambda_{1,\beta}|\Omega|.
$$
By using Lemma \ref{lemma2.2}, we get 
$$
\int_{B_1}|\nabla u_\beta|^2dx\to 0,\quad \text{as}\ \beta\to0.
$$
Thus, there exists some constant $C$ such that $u_\beta\to C$ as $\beta\to 0$. Since $0<u_\beta\leq 1$, we must have $0\leq C\leq 1$. If $C<1$, then $\sup u_\beta < 1$ for $\beta>0$ small enough. This contradicts the fact that $\|u_\beta\|_\infty=1$. Hence, we prove that $u_\beta\to 1$ as $\beta\to 0$.
\qed

\subsection{The case $p>1$}\label{s5}
It follows from the proof Theorem \ref{lemma2.5} and Lemma \ref{lemma2.7} that any positive solutions to equation \eqref{eq0.1} converge to 0 uniformaly in $\Omega$ as $\beta\to 0$ if $1<p<2^*-1$. Similarly, utilizing Lemma \ref{lemma2.8} and Lemma \ref{lemma2.6} we can see that the ground state solution converge to 0 uniformaly in $\Omega$ for critical case, and the positive radial solution converge to 0 uniformaly in $\Omega$ for all $p>1$ respectively.
It is natural to ask:  \\
{\em what is the exact asymptotic behavior of these positive solutions?}

In the following, we will prove Theorem \ref{Theorem1.4}.

{\bf Proof of Theorem \ref{Theorem1.4}:}
The proof is standard, we can argue as Theorem \ref{Theorem1.2}. For the convenience of readers, we provide a brief proof here. Firstly, we recall that $\tilde{u}_\beta$ is a solution of equation \eqref{0.1}. Then it follows from Theorem \ref{lemma2.5}  that $u_\beta\to 0$ as $\beta\to 0$. Thus, we can get $\|u_\beta\|^{p-1}_\infty\to 0$ since $p>1$. Moreover by the fact that $0<\tilde{u}_\beta\le 1$, we can deduce that $\tilde{u}_\beta\to \tilde{u}_0$ in $C^2_{loc}(\bar{\Omega})$ as $\beta\to0$. Since $\tilde{u}_0$ solves equation \eqref{0.2}, we can derive  $\tilde{u}_0=1$ by contradiction. Finally, utilizing \eqref{0.4} we can derive  the asymptotic behavior of $u_\beta$.
\qed

\bibliographystyle{abbrv}
\bibliography{ChenGrossiLi.bib} 

\end{document}